\documentclass[11pt,letterpaper]{amsart}
\usepackage{amsmath}
\usepackage{amstext}
\usepackage{amssymb}
\usepackage{amsfonts}
\usepackage{enumerate}
\usepackage{mathrsfs}
\usepackage{braket}
\usepackage{color}
\usepackage[all]{xy}
\usepackage{url}
\usepackage{bbm}

\numberwithin{equation}{section}

\newtheoremstyle{note}
{1em}
{1em}
{}
{}
{\bfseries}
{:}
{.5em}
{}

\newtheorem{theorem}{Theorem}[section]
\newtheorem{lemma}[theorem]{Lemma}
\newtheorem{proposition}[theorem]{Proposition}
\newtheorem{corollary}[theorem]{Corollary}

\theoremstyle{note}

\newtheorem{question}[theorem]{Question}
\newtheorem{definition}[theorem]{Definition}

\newtheorem*{acknowledgments}{Acknowledgments}

\newcommand{\N}{{\mathbb{N}}}
\newcommand{\R}{{\mathbb{R}}}
\newcommand{\C}{{\mathbb{C}}}

\newcommand{\K}{{\mathbb{K}}}
\newcommand{\n}[1]{ \left\|#1\right\| }
\newcommand{\OH}{\mathrm{OH}}

\newcommand{\U}{{\mathcal{U}}}

\DeclareMathOperator{\Rea}{Re}

\DeclareMathOperator{\cb}{cb}

\title{Completely coarse maps are $\R$-linear}

\author[B. M. Braga]{Bruno M. Braga}

\address[B. M. Braga]{University of Virginia, $141$ Cabell Drive, Kerchof Hall, P.O. Box $400137$, Charlottesville, USA}
\email{demendoncabraga@gmail.com}

\author[J. A. Ch\'avez-Dom\'inguez]{Javier Alejandro Ch\'avez-Dom\'inguez}
\address[J. A. Ch\'avez-Dom\'inguez]{Department of Mathematics, University of Oklahoma, Norman, OK 73019-3103,
USA} \email{jachavezd@ou.edu}

\thanks{The second-named author was partially supported by NSF grant DMS-1900985.}


\begin{document}

\maketitle

\begin{abstract}
A map between operator spaces is called completely coarse if the sequence of its aplifications is equi-coarse. We prove that all completely coarse maps  must be $\R$-linear.       On the opposite direction of this result, we introduce a notion of embeddability between operator spaces and show that this notion is  strictly weaker than complete $\R$-isomorphic embeddability (in particular, weaker than complete $\C$-isomorphic embeddability). Although weaker, this notion is strong enough for some applications. For instance, we show that if an infinite dimensional operator space $X$ embeds in this weaker sense into Pisier's operator space $\OH$, then $X$ must be completely isomorphic to $\OH$.  
\end{abstract}

\section{Introduction}

Throughout, $\K$ is either $\R$ or $\C$.
Recall that a \emph{(concrete) $\K$-operator space} is a closed subspace $X$ of  $B(H)$, where $H$ is a $\K$-Hilbert space and $B(H)$ denotes the $\mathrm{C}^*$-algebra of all bounded $\K$-linear operators on $H$.
This naturally induces a norm on each of the spaces $M_n(X)$ of $n \times n$ matrices with entries in $X$, inherited from the identification $M_n(B(H)) = B(H^n)$ where $H^n$ is the Hilbert space given by the direct sum of $n$ copies of $H$ endowed with the $\ell_2$-sum norm.
The theory of $\C$-operator spaces is by now very well-developed as a (noncommutative) quantization of Banach spaces, see e.g. \cite{Effros-Ruan-book,Pisier-OS-book}.
The theory for $\R$-operator spaces has not received as much attention, but a number of important results from the complex case also hold in the real one: see \cite{Ruan-real-OS,Ruan-complexifications,Sharma}.

If $X$ and $Y$ are $\K$-operator spaces and $f : X \to Y$ is a (not necessarily linear) function, its \emph{amplifications} are the functions $f_n : M_n(X) \to M_n(Y)$ that are defined by applying $f$ entry-wise.
The morphisms between $\K$-operator spaces are the \emph{completely bounded maps}, that is, $\K$-linear maps $f : X \to Y$ with finite \emph{completely bounded norm} $\n{f}_{\cb} := \sup_n \n{f_n}$.
Amplifications of nonlinear maps have also been considered for a long time: classical results of Schoenberg and Rudin \cite{Schoenberg,Rudin} characterize functions on $\R$ which are \emph{completely positive}, that is,  whose amplifications map positive-semidefinite matrices to positive-semidefinite matrices.

In the last 20 years or so, the  nonlinear   geometry of Banach spaces has increasingly attracted the attention of many Banach space theorists (see, e.g. the surveys \cite{KaltonNonlinear2008,Godefroy-Lancien-Zizler} and references therein),  hence  it is natural to search of an ``appropriate'' theory for  the nonlinear geometry of operator spaces. With this goal in mind, Sinclair together with the first named author   have considered amplifications of various types of nonlinear maps from a metric perspective  \cite{Braga-Sinclair}:
 a map $f : X \to Y$ between operator spaces is called \emph{completely Lipschitz (resp. completely uniformly continuous, completely coarse)} if the family of all its amplifications is equi-Lipschitz (resp. equi-uniformly continuous, equi-coarse), see Section \ref{section-preliminaries} for the precise definitions.
 
The objective of this   note is to show that the   approach  of \cite{Braga-Sinclair} does not  give rise to a genuine nonlinear theory of operator spaces. Precisely, the following is the main theorem of this paper.

\begin{theorem}\label{thm-completely-coarse-implies-R-linear}
Let $X$ and $Y$ be $\K$-operator spaces, and
let $f : X \to Y$ be completely coarse. If $f(0)=0$, then $f$ is $\R$-linear.
\end{theorem}

Therefore, if $\K=\R$, then the class of completely coarse maps is precisely the class of completely bounded operators and, if $\K=\C$, the complex structures of the operator spaces is the only restriction for  completely coarse maps to be linear.\footnote{The first named author   would like to say that since the new findings contained in these notes make \cite{Braga-Sinclair} obsolete, he and Sinclair   decided to leave  \cite{Braga-Sinclair} unpublished.}

Continuing the search for an ``appropriate'' framework for a theory of the  nonlinear geometry of operator spaces, in the second half of this paper, we introduce the notion of \emph{almost complete coarse embeddability} between $\K$-operator spaces and show that this is indeed  weaker than complete $\R$-isomorphic embeddability (see Definition   \ref{DefiAlmostCompCoarEmb} and Theorem  \ref{Prop62}). Although this notion is strictly weaker than   complete $\R$-isomorphic embeddability, it is not clear whether this is the ``correct'' weakening to look at. Precisely, although we show that  almost complete coarse embeddability is a strong enough notion in order to give us interesting applications (see Proposition  \ref{PropAlmostComplCoarseEmbImpliesRlinear} and Theorem \ref{CorOH}), we still do not know if this is an genuinely nonlinear notion (see Question \ref{Question1}). 

Nevertheless, not only we show that almost complete coarse embeddability  is strictly weaker than complete $\R$-isomorphic embeddability, but we also show that this notion of embeddability is strong enough to recover   complete $\R$-isomorphic results (which implies that almost complete coarse embeddability is a worthstudying notion).  For instance, we show that almost complete coarse embeddability between $\K$-operator spaces $X$ and $Y$ imply that $X$ completely $\R$-isomorphically embeds into any  ultrapower $Y^\N/\mathcal U$, for an nonprincipal ultrafilter $\mathcal U$ on $\N$ (see Proposition  \ref{PropAlmostComplCoarseEmbImpliesRlinear}). As an application,   we obtain that, despite the fact that recovering $\C$-linearity from a $\R$-linear map is usually not possible,\footnote{Recall,  Bourgain showed that  there are nonisomorphic $\C$-Banach spaces which are isomorphic as $\R$-Banach spaces \cite{Bourgain1986PAMS} and Ferenczi strengthened this result showing that  there are $\C$-Banach spaces which are $\R$-linearly isomorphically to each other but so that one does not $\C$-linearly embed into   the other \cite{Ferenczi2007Advances}.} the scenario is different if one considers Pisier's operator space $\mathrm{OH}(I)$. Precisely, we prove the following: 

  \begin{theorem}\label{CorOH}
Let $I$ be an index set. If a $\C$-operator space $X$ almost completely coarsely embeds into $\OH(I)$, then $X$ is completely $\C$-isomorphic to $\OH(J)$ for some index set $J$.
 \end{theorem}

In particular, the theorem above implies that if an infinite dimensional $\C$-operator space almost completely coarsely embeds into $\OH$, then it must be completely $\C$-isomorphic to $\OH$.

\section{Preliminaries}\label{section-preliminaries}

  Throughout this paper, either   $\K=\R$ or $\K=\C$. The reader can either follow our note with an unfixed  $\K$ in mind or pick their favorite field and proceed. We made this presentation choice since the paper deals in essence with when certain maps are automatically $\R$-linear and when certain nonlinear embeddability notions can be replaced by $\C$-linear notions. We will always be clear both when we are working with a specific choice of $\K$ or with an unfixed $\K$.
   
A subset of a $\K$-operator space is called a  \emph{$\K$-operator metric space}. Equivalently, $ E$ is a  $\K$-operator metric space if $ E$ is a subset of $ B(H)$, for some $\K$-Hilbert space $H$. Given  $n\in\N$, we consider  $M_n( E)$ with the   natural norm $\|\cdot\|_{M_n(E)}$  given by the canonical inclusion $M_n(E)\subset B(H^n)$ (in particular,  this norm defines	 a metric on   $M_n(E)$). For simplicity, we often  write $\|\cdot\|_n$ for $\|\cdot\|_{M_n(E)}$. Elements in $M_n(E)$ are denoted by $[x_{ij}]_{ij=1}^n$, or simply $[x_{ij}]_{ij}$ (or even $[x_{ij}]$).

\subsection{Nonlinear maps between $\K$-operator spaces}
If $ E$ and $ F$ are $\K$-operator metric spaces and  $f: E\to  F$ is a map, we can still consider the amplifications $f_n:M_n( E)\to M_n( F)$.
The map $f$ is a \emph{complete isometry}, or a \emph{complete isometric embedding}, if $f_n$ is an isometry for all $n\in\N$, and $ E$ and $ F$ are called \emph{completely isometric} if there exists a bijective complete isometry $ E\to  F$.

Recall, if  $(A,d_A)$ and $(B,d_B)$ are metric spaces, and  $f: A\to B$ is a map,   its \emph{modulus of uniform continuity}  is given by  \[\Delta_f(t) =\sup\{d_B(f(x), f(y)) \mid d_A(x,y)\leq t\}\]
for all $t\geq 0$. When working with $\K$-operator spaces, \cite{Braga-Sinclair} has introduced a complete version of this modulus:
given  $\K$-operator metric spaces $ E$ and $ F$ and a map $f:  E\to  F$, we define its \emph{complete modulus of uniform continuity} by 
\[
\Delta^{cb}_f(t) = \sup_{n\in \N} \Delta_{f_n}(t)
\]
for all $t\geq 0$.
Given  $\K$-operator metric spaces $ E$ and $ F$, and a map $f:  E\to  F$, we say that $f$ is:
\begin{enumerate}[(a)]
\item \emph{completely Lipschitz} if   $\sup_{t>0}\Delta^{\cb}_f(t)/t<\infty$,
 
\item \emph{completely uniformly continuous} if $\lim_{t\to 0}\Delta^{\cb}_f(t)=0$, 
\item \emph{completely coarse} if for all $t>0$ it holds that $\Delta^{\cb}_f(t)<\infty$,
\item a \emph{complete Lipschitz embedding} if $f^{-1}:\mathrm{Im}(f(E))\to E$ exists and both $f$ and $f^{-1}$ are  completely Lipschitz,
\item a \emph{complete uniformly continuous  embedding} if $f^{-1}:\mathrm{Im}(f(E))\to E$ exists and both $f$ and $f^{-1}$ are  completely uniformly continuous,
\end{enumerate}
Moreover, we say that maps  $f,g:E\to F$ are \emph{completely close} if 
\[\sup_{n\in\N}\sup_{x\in X}\|f_n(x)-g_n(x)\|_{M_n(F)}<\infty\]
and we say that $f$ is:
\begin{enumerate}[(a)]
\item[(f)] a \emph{complete coarse embedding} if $f$ is completely coarse and  there is  a completely coarse map $h:\mathrm{Im}(f(E))\to E$ so that  $h\circ f$ and  $f\circ h$ are completely close to $\mathrm{Id}_E$ and $\mathrm{Id}_{f(E)}$, respectively.
\end{enumerate}


The following is elementary and it is the operator space version of \cite[Lemma 	1.4]{KaltonNonlinear2008}.

\begin{proposition}\label{Prop1}
Let $H$ be a $\K$-Hilbert space, and $E$ and $F$ be $\K$-vector subspaces of $B(H)$. Then a map  $f : E \to F$  is   completely
coarse if and only if  there is $C>0$ so that 
\[\|f([x_{ij}])-f([y_{ij}])\|_n\leq C\|[x_{ij}]-[y_{ij}]\|_n+C\]
for all $n\in\N$ and all $[x_{ij}]\in M_n(E)$. Moreover, if $f$ is completely uniformly continuous, then $f$ is completely coarse.\qed
\end{proposition}

\subsection{Hadamard matrices} Recall that a Hadamard matrix is a square matrix whose entries are either $+1$ or $-1$ and whose rows are mutually orthogonal.
It is easy to see that an $n\times n$ Hadamard matrix has norm $\sqrt{n}$, 
since dividing the matrix by $\sqrt{n}$ yields an orthogonal matrix.
It can be shown that Hadamard matrices of arbitrarily large sizes do exist, for example by using Sylvester's construction:
define
\[
A_2 = \begin{pmatrix} 1 &1\\ 1 &-1\end{pmatrix}, \qquad A_{2^{k+1}} =   \begin{pmatrix} A_{2^k} &A_{2^k}\\ A_{2^k} &-A_{2^k}\end{pmatrix} = A_2 \otimes A_{2^k}.
\]

We denote the $n\times n$ matrix all of whose entries are $1$ by $\mathbbm{1}_n$.
It is easy to see that $\n{\mathbbm{1}_n} \ge n$, 
by applying $\mathbbm{1}_n$ to a vector of all 1's.

\subsection{Obtaining linearity} The following lemma is what will allow us to conclude linearity.
The proof is a standard functional equation argument as in, e.g. \cite[Chap. 11]{Engel}, but we provide a sketch for completeness.

\begin{lemma}\label{lemma-functional-equation}
Let $X$ be a normed $\R$-vector space, and let $f : X \to X$ be a function such that $f(0)=0$, $f$ is bounded on a neighborhood of $0$, and $f\big(\tfrac{1}{2}(x+y) \big) = \tfrac{1}{2}\big( f(x)+f(y) \big)$ for all $x,y \in X$. Then $f$ is $\R$-linear.
\end{lemma}

\begin{proof}
Setting $y=0$ yields $f(x/2)= f(x)/2$.
It follows that $f(x+y)=f(x)+f(y)$ for any $x,y\in X$,
i.e. $f$ is additive.
Setting $y=x$ yields $f(2x) = 2f(x)$; by induction we get $f(nx) = nf(x)$ for all $x\in X$ and $n\in \N$.
Let $r,M>0$ be such that $\n{f(x)} \le M$ whenever $\n{x}\le r$.
Fix $x \in X$ with $\n{x} \le r$, and define $g_x : \R \to X$ by $g_x(t) = f(tx)-tf(x)$.
Note that $g_x$ is bounded by $2M$ on $[-1,1]$.
Moreover, by the additivity of $f$ we have
$g_x(t+1) = f(tx+x)-(t+1)f(x) = g_x(t)$, i.e. $g_x$ is periodic with period 1 and hence $g_x$ is bounded. If there existed a $t_0\not=0$ such that $g_x(t_0) \not=0$, it would follow that the sequence given by $g_x(nt_0) = ng_x(t_0)$, $n\in\N$ is unbounded, a contradiction. Therefore, $f(tx)=tf(x)$ for all $t\in \R$,
and thus we conclude $f$ is $\R$-linear.
\end{proof}

\subsection{Miscellaneous facts on operator theory} The following Proposition is well-known in the complex case, and the real one can be proved similarly:
it is an easy consequence of the fact that for a matrix  $(a_{ij}) \in M_n(\K)$,
\begin{equation}\label{eqn-norm-of-matrix}
\n{(a_{ij})}_{M_n(\K)} = \sup \left\{ \bigg| \sum_{i,j=1}^n a_{ij} w_i z_j \bigg|  \,:\, w, z \in \K^n, \n{w}_{\ell_2^n},\n{z}_{\ell_2^n} \le 1  \right\}.    
\end{equation}

\begin{proposition}
Let $X$ be an $\R$-operator space, and $\phi : X \to \R$ a continuous $\R$-linear functional. Then $\phi$ is completely bounded.\qed
\end{proposition}

Finally, a small remark regarding certain vector-valued matrices:
if $X \subseteq B(H)$ is a $\K$-operator space, $x \in X$, and $A \in M_n(\K)$,
then
\[
\n{ A \otimes x  }_{M_n(X)} = \n{A}_{M_n(\K)} \n{x}_X.
\]

\section{The main result}

We prove Theorem \ref{thm-completely-coarse-implies-R-linear} in this section.  In order to illustrate the strategy of our proof, we take a moment to show an example that encapsulates the essence of the argument.

\begin{proposition}\label{prop-abs-value}
Let $f : \K \to \K$ be given by $f(x) = |x|$. Then $f$ is not completely Lipschitz.
\end{proposition}

\begin{proof}
Consider Hadamard matrices $\big(A_{2^k}\big)_{k=1}^\infty$ as above.
Observe that $f_{2^k}\big(A_{2^k}\big) = \mathbbm{1}_{2^k}$.
Since $\n{A_{2^k}} = \sqrt{2^k}$ and $\n{\mathbbm{1}_{2^k}} \ge 2^k$, it follows that the Lipschitz constant of $f_{2^k}$ is at least $\sqrt{2^k}$; therefore, $f$ is not completely Lipschitz.
\end{proof}

We now prove our main result.

\begin{proof}[Proof of Theorem \ref{thm-completely-coarse-implies-R-linear}]
Let $X$ and $Y$ be $\K$-operator spaces and $f:X\to Y$ be a completely coarse map so that $f(0)=0$.  Fix $x_0\in X$ and $h\in X\setminus\{0\}$.
Let $\phi : X \to \K$ be a $\K$-linear continuous functional such that $\phi(h)=1$, and set $y_0 = \big[ f(x_0-h)-f(x_0+h) \big]/2$.
Define $g:X\to Y$ by $g(x)= f(x) + \phi(x)y_0$.
Observe that $g(x_0 +h) = g(x_0 -h)$, and $g$ is completely coarse, because so are $f$ and $\phi$.
By Proposition \ref{Prop1}, there exists a constant $C>0$ such that for any $n\in\N$ and  $[x_{ij}]_{ij},[y_{ij}]_{ij} \in M_n(X)$,
\begin{equation}\label{eqn-coarse-Lip}
    \n{ \big[ g(x_{ij})-g(y_{ij}) \big]_{ij} }_{M_n(Y)} \le C \n{ [x_{ij}-y_{ij}]_{ij}}_{M_n(X)} + C.
\end{equation}

Once again, consider Hadamard matrices $ (A_{2^k} )_{k=1}^\infty$ as above.
For each $k\in\N$, write $A_{2^k} =  [ a^k_{i,j}  ]_{i,j=1}^{2^k}$.
Consider the matrices in $M_{2^k}(X)$ given by
\[
\mathbbm{1}_{2^k}\otimes x_0 +  A_{2^k} \otimes h =  \big[ x_0 + ha^k_{ij} \big]_{i,j=1}^{2^k}, \qquad \mathbbm{1}_{2^k} \otimes x_0 =  \big[ x_0 \big]_{i,j=1}^{2^k}.
\]
By \eqref{eqn-coarse-Lip},
\[
\n{ \big[ g(x_0 + ha^k_{ij})-g(x_0) \big]_{i,j=1}^{2^k} }_{M_{2^k}(Y)} \le C \n{ A_{2^k} \otimes h }_{M_{2^k}(X)} + C.
\]
Since $g(x_0 +h) = g(x_0 -h)$, this means
\[
\n{g(x_0+h)-g(x_0)} \cdot\n{ \mathbbm{1}_{2^k} }_{M_{2^k}(\K)} \le C \cdot \n{h} \cdot \n{ A_{2^k} }_{M_{2^k}(\K)} + C,
\]
which yields
\[
\n{g(x_0+h)-g(x_0)}2^{k} \le C \n{h} \sqrt{2^k} + C.
\]
Since this must hold for all $k\in\N$, we conclude that  $g(x_0+h)-g(x_0)=0$, so
$g(x_0) = g(x_0+h)$.  That is,
$f(x_0)+\phi(x_0)y_0 = f(x_0+h) + \phi(x_0+h)y_0$, from where
\[
f(x_0) = f(x_0+h) +y_0 = f(x_0+h) + \frac{f(x_0-h)-f(x_0+h)}{2},
\]
which implies $f(x_0) = \tfrac{1}{2}\big( f(x_0+h)+f(x_0-h) \big)$.
This means that $f$ satisfies the conditions of Lemma \ref{lemma-functional-equation} (note that since $f$ is completely coarse, in particular it is bounded on a neighborhood of $0$), and therefore $f$ is $\R$-linear.
\end{proof}

Notice that  complete coarseness is a weaker property than both complete uniform continuity and complete Lipschitzness (see Proposition \ref{Prop1}). Therefore,  we can    get the following corollary of Theorem \ref{thm-completely-coarse-implies-R-linear}.

\begin{corollary}
Let $X$ and $Y$ be $\K$-operator spaces, and $f:X\to Y$ be a map with $f(0)=0$. The following are equivalent:

\begin{enumerate}
\item\label{Item1} $f$ is a  complete coarse  embedding,
\item $f$ is a  complete uniformly continuous   embedding,
\item $f$ is a  complete Lipschitz  embedding, and 
\item\label{Item4} $f$ is a  complete $\R$-isomorphic embedding.
\end{enumerate}
\end{corollary}

\begin{proof}
By the comments preceding the corollary, we only need to show that \eqref{Item1} implies \eqref{Item4}. For that, suppose  $f:X\to Y$ is a  complete coarse  embedding. Then, by Theorem \ref{thm-completely-coarse-implies-R-linear}, $f$ is $\R$-linear, hence $f$ is a completely bounded operator. In particular, $F=f(X)$ is a $\K$-linear subspace of $Y$. As $f$ is a complete coarse embedding, there is a completely coarse map $g:F\to X$  so that $f\circ g$ and $g\circ f$ are (completely) close to $\mathrm{Id}_{F}$ and $\mathrm{Id}_X$, respectively.

Notice that although Theorem \ref{thm-completely-coarse-implies-R-linear} is stated for $\K$-operator spaces, i.e., for \emph{complete} $\K$-linear subspaces of $B(H)$, completeness of the spaces are not used in its proof.\footnote{Notice that both Proposition \ref{Prop1} and Lemma \ref{lemma-functional-equation} hold for   $\R$-vector subspaces of $B(H)$.} Therefore,  Theorem \ref{thm-completely-coarse-implies-R-linear} implies that $g$ is also $\R$-linear,  so both $f\circ g$ and $g\circ f$ are  $\R$-linear. Then, $\R$-linearity and closeness of those maps to $\mathrm{Id}_{F}$ and $\mathrm{Id}_X$, respectively, imply that $f\circ g=\mathrm{Id}_{F}$ and $g\circ f=\mathrm{Id}_X$. This implies that $F$ is indeed a $\K$-operator space and that $f$ is a complete $\R$-isomorphic embedding.
\end{proof}

\section{A weaker notion of embeddability}
In this section, we introduce a notion of nonlinear embeddability between operator spaces which is strictly weaker than complete $\R$-isomorphic embeddability but which it is still strong enough  so that it  preserves some features of the $\R$-operator space structure --- and in some cases even the $\C$-operator space structure.

\begin{definition}\label{DefiAlmostCompCoarEmb}
Let $X$ and $Y$ be $\K$-operator spaces. We say that $X$ \emph{almost completely coarsely embeds into} $Y$ if there are functions $\omega,\rho:[0,\infty)\to [0,\infty)$ so that $\lim_{t\to\infty}\rho(t)=\infty$ and a sequence of maps $(f^n:X\to Y)_n$ so that 
\[\rho(\|[x_{ij}]-[y_{ij}]\|_n)\leq \|f^n([x_{ij}])-f^n([y_{ij}])\|_n\leq \omega(\|[x_{ij}]-[y_{ij}]\|_n)\]
for all $n\in\N$ and all $[x_{ij}],[y_{ij}]\in M_n(X)$. If the maps $(f^n)_n$ are $\K$-linear, we say that $X$ \emph{almost completely $\K$-isomorphically embeds into }Y.
\end{definition}

\begin{theorem}\label{Prop62}
There are $\C$-operator spaces $X$ and $Y$ so that $X$ almost completely $\C$-isomorphically embeds into $Y$, but so that  $X$ does not completely $\R$-isomorphically embed into $Y$.
\end{theorem}

In the proof below, we use the minimal operator space structure of a Banach space. Precisely, if $E$ is a $\K$-Banach space, there is a canonical $\K$-linear isometric embedding of $E$ into the $\mathrm{C}^*$-algebra of continuous functions on $(B_{E^*},\sigma(E^*,E))$. This inclusion induces an operator space structure on $E$ and we call this operator space $\mathrm{MIN}(E)$. If   $F$ is a $\K$-operator space and $u:F\to \mathrm{MIN}(E)$ is a $\K$-linear bounded map, then $u$ is completely bounded, moreover, $\|u\|_{\mathrm{cb}}=\|u\|$ (see \cite[Subsection 1.2.21]{BlecherLeMerdy2004}). Furthermore, if $u:F\to \mathrm{MIN}(E)$ is a $\R$-linear bounded map, then analogous arguments show that $u$ is also completely bounded, in fact $\|u\|_{\mathrm{cb}}\leq 4\|u\|$. 

\begin{proof}[Proof of Theorem \ref{Prop62}]
Let $\ell_2(\C)$ be the $\C$-Hilbert space of square summable functions $\N\to \C$ and let $R\subset B(\ell_2(\C))$ be the Hilbertian row $\C$-operator space, i.e.,  \[R=\{a\in B(\ell_2(\C))\mid \forall m\neq 1, \langle ae_n,e_m\rangle=0\},\]
where $(e_n)_n$ is the standard unit basis of $\ell_2(\C)$. For each $n\in\N$, let $\mathrm{MIN}_n(R)$ be the $\C$-operator space considered in \cite[Section 2]{OikhbergRicard2004MathAnn}. Precisely, $\mathrm{MIN}_n(R)$ is the $\C$-Banach space $R$ with the following $\C$-operator space structure: for all $m\in\N$ and all $x\in M_m(\mathrm{MIN}(R))$  we have  
\[\|x\|_{M_m(\mathrm{MIN}(R))}=\sup\Big\{\|(\mathrm{Id}_{M_m}\otimes u)x\|_{{mn}}\mid u\in\mathrm{CB}(R,M_n(\C)), \|u\|_{\mathrm{cb}}\leq 1\Big\}.\]
Then, the identity $R\to \mathrm{MIN}_n(R)$ induces $\C$-linear isometries $M_k(R)\to M_k(\mathrm{MIN}_n(R))$ for all $k\leq n$ (see \cite[Lemma 2.1]{OikhbergRicard2004MathAnn}). 

Notice that $\mathrm{MIN}_n(R)$ is completely $\C$-isomorphic to the $\C$-operator space $\mathrm{MIN}(R)$ for all $n\in\N$. Indeed, let $\mathrm{I}:\mathrm{MIN}(R)\to \mathrm{MIN}_n(R)$ be the identity. Then, by the properties of $\mathrm{MIN}(R)$ mentioned above,  $\mathrm{I}^{-1}$ is completely bounded. Moreover,   $\mathrm{I}$ is also completely bounded by \cite[Proposition 2.2]{OikhbergRicard2004MathAnn}.

Let \[Y=\bigoplus_{n\in\N} \mathrm{MIN}_n(R)\] be the $\infty$-direct sum operator space (see \cite[1.2.17]{BlecherLeMerdy2004}). Clearly, $R$ almost completely $\C$-isomorphically  embeds into $Y$. 	

We are left to notice that $R$ does not completely $\R$-isomorphically embed into $Y$. Suppose otherwise and let  $u:R\to Y$ be such embedding. For each $n\in\N$, let $p_n:Y\to \mathrm{MIN}_n(R)$ be the canonical projection. If there exists $n\in\N$ so that  $p_nu:R\to \mathrm{MIN}_n(R)$ is not completely strictly singular,\footnote{An $\R$-linear map $u:X\to Y$ between $\C$-operator space is \emph{completely strictly singular} if its restriction to any infinite dimensional $\C$-vector subspace of $X$ is not a complete $\R$-linear isomorphic embedding.} then there is an infinite dimensional $\C$-vector subspace of $R$ which  completely $\R$-isomorphically embeds into $\mathrm{MIN}(R)$. As $R$ is a  homogeneous $\C$-operator space (see \cite[Section 9.2]{Pisier-OS-book} and \cite[Lemma 2.5]{OikhbergRicard2004MathAnn}),   \cite[Proposition 9.2.1]{PisierOH1996} implies that all $\C$-vector subspaces of $R$   are  completely $\C$-linearly isometric to $R$. So we conclude that   $R$ completely $\R$-linearly isomorphically embeds  into $\mathrm{MIN}(R)$. By the comments preceding this proof, this implies that  every $\R$-linear bounded map $E\to R$ is completely bounded; contradiction.

Therefore,   $p_nu$ is completely strictly singular for all $n\in\N$. Since $p_nu$ is completely strictly singular for all $n\in\N$, a simple sliding hump argument gives us a contradiction to the complete  $\R$-linear embeddability of $R$ into $Y$, we leave the details to the reader.
\end{proof}

\begin{question}\label{Question1}
Is almost complete coarse embeddability strickly weaker than almost complete $\R$-isomorphic embedability? I.e., are there $\K$-operator spaces $X$ and $Y$ so that $X$ almost completely coarsely embeds into $Y$, but $X$ does not almost completely $\R$-isomorphically embeds into $Y$?
\end{question}

The next result shows that  this weaker form of complete embeddability
 is already strong enough so that the $\R$-operator space local structure of $Y$ passes to $X$. Given a $\K$-operator space $Y$ and an ultrafilter $\U$ on $\N$, we let $Y^\N/\U$ be the $\K$-operator ultraproduct space (see \cite[Section 1.2.31]{BlecherLeMerdy2004}).

\begin{proposition}\label{PropAlmostComplCoarseEmbImpliesRlinear}
If a $\K$-operator space $X$ almost completely coarsely embeds into an $\K$-operator space $Y$, then $X$ completely $\R$-isomorphically  embeds into $Y^\N/\U$ for any nonprincipal ultrafilter $\U$ on $\N$.
\end{proposition}

\begin{proof}
Let $(f^n)_n$ be a sequence which witnesses that $X$ almost completely coarsely embeds into  $Y$ and let $\U$ be a nonprincipal ultrafilter on $\N$. Without loss of generality, assume that $f^n(0)=0$ for all $n\in\N$. Define $F:X\to Y^\N/\U$ by letting $F(x)=[(f^n(x))_n]$ for all $x\in X$. Since $\|f^n(x)\|\leq \omega(\|x\|)$ for all $x\in X$, this map is well defined. 

Given $k\in\N$ and $[x_{ij}],[y_{ij}]\in M_k(X)$, for any $n\geq k$, denote by $[\bar x_{ij}],[\bar y_{ij}]$ the image of $[x_{ij}]$ and $[y_{ij}]$ under the canonical isometry $M_k(X)\to M_n(X)$. Then we have that 
\begin{align*}
\|f^n([x_{ij}])-f^n([y_{ij}])\|_k&=\|f^n([\bar x_{ij}])-f^n([\bar y_{ij}])\|_n\\
&\leq \omega(\|[\bar x_{ij}]-[\bar y_{ij}]\|_n)\\
&=\omega(\|[x_{ij}]-[y_{ij}]\|_k).
\end{align*}
Since $\U$ is nonprincipal, this implies that 
\[\|F([x_{ij}])-F([y_{ij}])\|_k \leq \omega(\|[x_{ij}]-[y_{ij}]\|_k)\]
for all $k\in \N$ and all $[x_{ij}],[y_{ij}]\in M_k(X)$. I.e.,   $F$ is completely coarse and Theorem \ref{thm-completely-coarse-implies-R-linear} implies that $F$ is $\R$-linear. Clearly, $\|F\|_{\mathrm{cb}}\leq \omega(1)$. 

Similarly as in the inequalities above, the nonprincipality of $\U$ gives that  
\[\|F([x_{ij}])-F([y_{ij}])\|_k\geq 
\rho(\|[x_{ij}]-[y_{ij}]\|_k)\]
for all $k\in \N$, and all $[x_{ij}],[y_{ij}]\in M_k(X)$. As $\lim_{t\to\infty}\rho(t)=\infty$, this implies that $F$ is injective. Moreover, $\lim_{t\to\infty}\rho(t)=\infty$ also implies that $F^{-1}$ is completely bounded. Hence, $F$ is a complete $\R$-isomorphic embedding.
\end{proof}

 \subsection{Application to $\OH$}
We now show that almost complete coarse embeddability into  Pisier's operator Hilbert space implies complete $\C$-isomorphic embeddability (see Theorem \ref{CorOH}). But first,   we  recall  the complexification of a real operator space. Given an $\R$-vector  space  $X$, we define the  \emph{complexification of $X$}, denoted by $X_\C$, as the  direct sum \[X_\C=X\oplus i X =\{x+iy\mid x,y\in X\},\]
where 
\[(x_1+iy_1)+(x_2+iy_2)=(x_1+x_2)+i(y_1+y_2)\]
for all $x_1,x_2,y_1,y_2\in X$,  and
\[(\alpha+i\beta)(x+iy)=(\alpha x-\beta y)+i(\beta x+\alpha y)\]
for all $\alpha,\beta\in \R$ and all $x,y\in X$. So $X_\C$ is a $\C$-vector space.

If $H$ is an $\R$-Hilbert space with inner product $\langle\cdot,\cdot\rangle$ and norm $\|\cdot\|$ given by this inner product, we define an inner product (which we still denote by  $\langle\cdot,\cdot\rangle$) on $H_\C$ by letting 
\[\langle x_1+iy_1,x_2+iy_2\rangle=\langle x_1,x_2\rangle
+\langle y_1,y_2\rangle+i\langle y_1,x_2\rangle-\langle x_1,y_2\rangle\]
for all $x_1,x_2,y_1,y_2\in H$. This inner product gives a norm $\|\cdot\|_\C$ on $H_\C$ so that 
\[\|x+i0\|_\C=\|x\|\ \text{ and } \ \|x+iy\|_\C=\|x-iy\|_\C\]
for all $x,y\in H$.  It is easy to see that 
\[B(H_\C)= B(H)\oplus i B(H)= B(H)_\C\]
and this gives us a natural $\C$-operator space structure on $ B(H)_\C$.

We conclude that   if $ X$ is an $\R$-operator space in $  B(H)$, then $  X_\C$ has a canonical $\C$-operator space structure inherited by the inclusion $  X_\C\subset  B(H)_\C=B(H_\C)$.  We refer the reader to \cite[Section 2]{Ruan2003IllinoisJourMath} for more details on this complexification.

Recall, given $\K$-operator spaces $X$ and $Y$, we write $X\oplus Y$ (resp. $X\oplus_1 Y$) to denote the $\ell_\infty$-sum (resp. $\ell_1$-sum) of $X$ and $Y$.  Given $p\in [1,\infty]$,  we write 
\[X\oplus_p Y=(X\oplus Y,X\oplus_1 Y)_{1/p},\]
where $(X\oplus Y,X\oplus_1 Y)_{1/p}$ is the complex interpolation space of $X\oplus Y$ and $X\oplus_1 Y$ (here we make the convention that $1/\infty=0$) --- see \cite[Section 2.7]{Pisier-OS-book} for the definition of  interpolation spaces.
 
Given a $\C$-Banach space $X$, $\overline{X}$ denotes the \emph{conjugate    of $X$}, i.e., $\overline {X}=X$ and the scalar multiplication on $\overline{X}$ is given by $\alpha x=\bar \alpha x$ for all $\alpha \in \C$ and all $x\in \overline{X}$. Then, given a $\C$-operator space $Y\subset B(H)$, $\overline{Y}$ denotes the \emph{conjugate operator space of $Y$}, i.e.,  $\overline{Y}=Y$ and the operator space structure on $Y$ is given by the canonical inclusion $\overline{Y}\subset \overline{B(H)}= B(\overline{H})$.

\begin{proposition}\label{PropYoplusYBar}
If a $\C$-operator space $X$ completely $\R$-isomorphically embeds into a $\C$-operator space $Y$, then $ X$ completely $\C$-isomorphically embeds into $Y \oplus_p \overline{Y}$ for any $p\in [1,\infty]$.
\end{proposition}
 
\begin{proof}
First notice that considering $X$ as an $\R$-operator space, we can look at its complexification $X_\C=X\oplus I X$ -- we use $I$ instead of $i$ here because $X$ is already a complex operator space. Then  the map 
\[x\in X\mapsto x+I(-ix)\in X_\C\]
is a complete $\C$-isomorphic embedding. Let $f:X\to Y$ be a complete $\R$-isomorphic embedding.  Since the map 
\[x+Iy\in X_\C\mapsto (f(x)+if(y),f(x)-if(y))\in Y\oplus_p \overline{Y}\]
is a complete $\C$-isomorphic embedding, we are done.
\end{proof} 
 
 We need one last result before proving Theorem \ref{CorOH}.  Given an index set $I$, let $\OH(I)$ be the operator Hilbert space introduced by  Pisier in 
  \cite[Theorem 1.1]{PisierOH1996}.
 
\begin{proposition}\label{PropOH}
Let $I$ be an infinite index set, then $\OH(I)$ is completely $\C$-linearly isometric to $\OH(I)\oplus_2 \overline{\OH(I)}$.
\end{proposition}
 
 \begin{proof}
Let $\ell_2(I,\C)$ denote the $\C$-Hilbert space of all square summable functions $I\to \C$, and let $(e_i)_{i\in I}$ denote its standard basis. Let $ R(I)$ and $ C(I)$ be the row and column $\C$-operator spaces of $\ell_2(I,\C)$, respectively, i.e., fix $i_0\in I$ and let
\[C(I)=\{a\in B(\ell_2(I,\C))\mid \forall i\neq i_0,\ \langle ae_{i},e_j\rangle=0\};\]
$R(I)$ is defined similarly (cf. Proof of Theorem \ref{Prop62}). By \cite[Corollary 2.6]{PisierOH1996},  we have that 
\[\OH(I)=( R(I),  C(I))_{1/2},\]
where $( R(I), C(I))_{1/2}$ is the interpolation space of $ R(I)$ and $ C(I)$ (we refer the reader to \cite[Chapter 2]{PisierOH1996} for precise definitions). The map 
\[(\alpha_j+i\beta_j)_{j\in I}\in \ell_2(I)\mapsto (\alpha_j-i\beta_j)_{j\in I}\in \overline{\ell_2(I)}\]
defines a complete $\C$-linear isometry, and this induces complete $\C$-linear isometries between $ R(I) $ an $
\overline{ R(I)}$, and $C(I)$ and $\overline{ C(I)}$. Hence, 
we have that 
\[\overline{\OH(I)}=\overline{(R(I),C(I))_{1/2}}=(\overline{ R(I)},\overline{C(I)})_{1/2}=( R(I), C(I))_{1/2}=\OH(I)\]
completely $\C$-linearly isometric.

Now notice that 
\begin{align*}
\overline{(\OH(I)\oplus_2\OH(I))^*}&=\overline{\OH(I)^*\oplus_2\OH(I)^*}\\
&=\overline{\OH(I)^*}\oplus_2\overline{\OH(I)^*}\\
&=\OH(I)\oplus_2\OH(I)
\end{align*}
completely $\C$-linearly isometrically. Hence, by the uniqueness property of  $\OH(I)$, the conclusion of the proposition  follows.
 \end{proof}

\begin{proof}[Proof of Theorem \ref{CorOH}]
Let $\U$ be a nonprincipal filter on $\N$. By Proposition \ref{PropAlmostComplCoarseEmbImpliesRlinear}, the hypothesis imply that $X$ completely $\R$-isomorphically embeds into $\OH(I)^\N/\U$. By \cite[Lemma 3.1(ii)]{PisierOH1996}, $\OH(I)^\N/\U$ is completely $\C$-linearly isometric to $\OH(L)$ for some set $L$. Hence by Proposition \ref{PropYoplusYBar}, we have that $X$ completely $\C$-isomorphically embeds into $\OH(L)\oplus_2\overline{\OH(L)}$. The conclusion then follows from Proposition \ref{PropOH}
\end{proof}

\begin{acknowledgments}
The first named author would like to thank Thomas Sinclair for introducing him to operator spaces and fruitful conversations. The first named author is also thankful to Valentin Ferenczi for useful suggestions about complexifications of operator spaces.
\end{acknowledgments}

\providecommand{\bysame}{\leavevmode\hbox to3em{\hrulefill}\thinspace}
\providecommand{\MR}{\relax\ifhmode\unskip\space\fi MR }
\providecommand{\MRhref}[2]{%
  \href{http://www.ams.org/mathscinet-getitem?mr=#1}{#2}
}
\providecommand{\href}[2]{#2}

\end{document}